\documentclass[11pt]{amsart}
\usepackage[utf8]{inputenc}
\usepackage{amsmath}
\usepackage{amsthm}
\usepackage{amssymb}
\usepackage{tikz-cd}
\usepackage[T1]{fontenc}
\usepackage{wasysym}
\usepackage{dsfont}
\usepackage[many]{tcolorbox}
\usepackage{soul}
\usepackage{xcolor}
\usepackage{appendix}
\usepackage{enumerate}
\usepackage{multicol}
\usepackage{fancyhdr}
\usepackage{parskip}
\usepackage{hyperref}
\usepackage{mathtools}
\usepackage{nameref}
\usepackage{thmtools}
\usepackage{thm-restate}
\usepackage{adjustbox}
\usepackage{cleveref}
\usepackage{tikz}
\usepackage{float}
\usepackage{orcidlink}

\hypersetup{
  colorlinks   = true, 
  urlcolor     = {blue!50!black}, 
  linkcolor    = {blue!50!black}, 
  citecolor   = {red!50!black} 
}

\newtheorem{theorem}{Theorem}[section]
\newtheorem*{theorem*}{Theorem}
\newtheorem{lemma}[theorem]{Lemma}

\newtheorem*{proposition*}{Proposition}
\newtheorem{corollary}[theorem]{Corollary}
\newtheorem*{corollary*}{Corollary}
\newtheorem*{fact*}{Fact}

\newtheorem{conjecture}[theorem]{Conjecture}

\theoremstyle{definition}
\newtheorem*{definition*}{Definition}
\newtheorem{definition}[theorem]{Definition}

\newtheorem*{question*}{Question}

\newtheorem*{notation*}{Notation}

\theoremstyle{remark}

\newcommand{\M}{\mathcal{M}}
\newcommand{\Th}{\mathrm{Th}}

\renewcommand{\L}{\mathcal{L}}
\newcommand{\ar}{\mathrm{ar}}

\newcommand{\restr}{\upharpoonright}

\newcommand{\C}{\mathcal{C}}

\newcommand{\Ind}[1]
{#1\setbox0=\hbox{$#1x$}\kern\wd0\hbox to 0pt{\hss$#1\mid$\hss} \lower.9\ht0\hbox to 0pt{\hss$#1\smile$\hss}\kern\wd0}

\newcommand{\notind}[1]
{#1\setbox0=\hbox{$#1x$}\kern\wd0
\hbox to 0pt{\mathchardef\nn=12854\hss$#1\nn$\kern1.4\wd0\hss}
\hbox to 0pt{\hss$#1\mid$\hss}\lower.9\ht0 \hbox to 0pt{\hss$#1\smile$\hss}\kern\wd0}

\newcommand{\N}{\mathbb{N}}

\newcommand{\Gaif}{\mathsf{Gaif}}

\newcommand{\Bfrak}{\ensuremath{\mathfrak{B}}}
\newcommand{\Cfrak}{\ensuremath{\mathfrak{C}}}

\newcommand{\Dcal}{\ensuremath{\mathcal{D}}}

\newcommand{\Ical}{\ensuremath{\mathcal{I}}} 
 
\newcommand{\Kcal}{\ensuremath{\mathcal{K}}} 

\newcommand{\Mcal}{\ensuremath{\mathcal{M}}}

\newcommand{\Pcal}{\ensuremath{\mathcal{P}}}

\newcommand{\Rcal}{\ensuremath{\mathcal{R}}}
\newcommand{\Scal}{\ensuremath{\mathcal{S}}}

\newcommand{\Psf}{\ensuremath{\mathsf{P}}}

\newcommand{\FPT}{\ensuremath{\mathsf{FPT}}}
\newcommand{\FO}{\ensuremath{\mathsf{FO}}}

\newcommand{\PP}{\mathbf{P}}

\usetikzlibrary{graphs.standard} 

\title{Monadic NIP in monotone classes of relational structures}

\author[S. Braunfeld]{Samuel {Braunfeld}\ \orcidlink{0000-0003-3531-9970}}
\address{Samuel {Braunfeld}, Computer Science Institute of Charles University (IUUK), Praha, Czech Republic}
\email{\href{sbraunfeld@iuuk.mff.cuni.cz}{sbraunfeld@iuuk.mff.cuni.cz}}

\author[A. Dawar]{Anuj {Dawar}\ \orcidlink{0000-0003-4014-8248}}
\address{Anuj {Dawar}, Department of Computer Science and Technology, University of Cambridge, UK}
\email{\href{anuj.dawar@cl.cam.ac.uk}{anuj.dawar@cl.cam.ac.uk}}

\author[I. Eleftheriadis]{Ioannis {Eleftheriadis}\ \orcidlink{0000-0003-4764-8894}}
\address{Ioannis {Eleftheriadis}, Department of Computer Science and Technology, University of Cambridge, UK}
\email{\href{ie257@cam.ac.uk}{ie257@cam.ac.uk}}

\author[A. Papadopoulos]{Aris {Papadopoulos}\ \orcidlink{0000-0001-7071-4277}}
\address{Aris {Papadopoulos}, School of Mathematics, Univesity of Leeds, UK}
\email{\href{mailto:mmadp@leeds.ac.uk}{mmadp@leeds.ac.uk}}

\thanks{The first author is supported by Project 21-10775S of the Czech Science Foundation (GA\v{C}R), European Union’s Horizon 2020 research and innovation programme (grant agreement No 810115 – Dynasnet). The third author is supported by a George and Marie Vergottis Scholarship awarded through Cambridge Trust, Onassis Foundation Scholarship, Robert Sansom Studentship. The fourth author is supported by a Leeds Doctoral Scholarship, from the University of Leeds.}

\subjclass[2020]{Primary: 03C13, 03C45 Secondary: 05C75, 68R05, 68Q27}

\keywords{Model theory, finite model theory, structural graph theory, model-checking}

\begin{document}

\maketitle

\begin{abstract}
We prove that for any monotone class of finite relational structures, the first-order theory of the class is NIP in the sense of stability theory if, and only if, the collection of Gaifman graphs of structures in this class is nowhere dense.  This generalises to relational structures a result previously known for graphs and answers an open question posed by Adler and Adler (2014).  The result is established by the application of Ramsey-theoretic techniques and shows that the property of being NIP is highly robust for monotone classes.  We also show that the model-checking problem for first-order logic is intractable on any class of monotone structures that is not (monadically) NIP. This is a contribution towards the conjecture of Bonnet et al.\ that the hereditary classes of structures admitting fixed-parameter tractable model-checking are precisely those that are monadically NIP.
\end{abstract}

\section{Introduction}

The development of stability theory in classical model theory, originating with Shelah's classification programme fifty years ago~\cite{She78,Bal17}, has sought to distinguish \emph{tame} first-order theories from \emph{wild} ones.  A key discovery is that combinatorial configurations serve as dividing lines in this classification.

Separately, in the development of finite model theory, there has been in interest in investigating \emph{tame} classes of finite structures.  Here tameness can refer to \emph{algorithmic tameness} meaning that algorithmic problems that are intractable in general may be tractable on a tame class; or it can refer to \emph{model-theoretic tameness} meaning that the class enjoys some desirable model-theoretic properties that are absent in the class of all finite structures.  See~\cite{Daw07} for an exposition of these notions of tameness.  The tame classes that arise in this context are based on notionns taken from the study of sparse graphs~\cite{NOdM12} and often extended to classes of relational structures beyond graphs by applying them to the \emph{Gaifman graphs} of such structures.

In the context of algorithmic tameness of sparse classes, this line of work culminated in the major result of Grohe et al.~\cite{nwdtract} showing that the problem of model-checking first-order sentences is fixed-parameter tractable ($\FPT$) on any class of graphs that is \emph{nowhere dense}.  This generalized a sequence of earlier results showing the tractability of the model-checking problem on classes of graphs satisfying other notions of sparsity.  Moreover, it is also known~\cite{KD09} that this is the limit of tractability for \emph{monotone} classes of graphs.  That is to say that (under reasonable assumptions) any monotone class of graphs in which first-order model checking is $\FPT$ is necessarily nowhere dense.  These results underline the centrality of the notion of nowhere density in the study of sparse graph classes.

A significant line of recent research has sought to generalize the methods and results on tame sparse classes of graphs to more general classes that are not necessarily sparse.  Interestingly, this has tied together notions of tameness arising in finite model theory and those in classical model theory.  Notions arising from stability theory play an increasingly important role in these considerations (see~\cite{PST18,GPT22}, for example).  Central to this connection is the realisation that for well-studied notions of sparseness in graphs, the first-order theory of a sparse class $\mathcal C$ is stable.  Thus, stability-theoretic notions of tameness, applied to the theory of a class of finite structures, generalize the notions of tameness emerging from the theory of sparsity.

A key result connecting the two directions is that a monotone class of finite graphs is stable if, and only if, it is nowhere dense.  This connection between stability and combinatorial sparsity was established in the context of infinite graphs by Podewski and Ziegler~\cite{podewskiziegler} and extended to classes of finite graphs by Adler and Adler~\cite{AdlerAdler2014}.  Indeed, for monotone classes of graphs, stability is a rather robust concept as the theory of such a class is stable if, and only if, it is NIP (that is, it does not have the independence property) and these conditions are in turn equivalent to it being monadically stable and monadically NIP (these notions are formally defined in Section~\ref{sec:prelim} below).

A question posed by Adler and Adler is whether their result can be extended from graphs to structures in any finite relational language.  We settle this question in the present paper by establishing Theorem~\ref{th:maintheorem} below.  In the following $\Gaif(\C)$ denotes the collection of Gaifman graphs of structures in the class $\C$.  Note that the extension from graphs to relational structures requires considerable combinatorial machinery in the form of Ramsey-theoretic results, which we detail in later sections.  We also relate the characterization to the tractability of the classes.  In summary, our key results are stated in the following theorem. 
\begin{theorem}\label{th:maintheorem}
Let $\C$ be a monotone class of structures in a finite relational language. Then the following are equivalent:
\begin{enumerate}
    \item $\C$ is NIP;
    \item $\C$ is monadically NIP; 
    \item $\C$ is stable; 
    \item $\C$ is monadically stable;
    \item $\Gaif(\C)$ is nowhere dense; and 
    \item (assuming $\mathsf{AW}[*]\neq \FPT$) $\C$ admits fixed-parameter tractable model-checking.
\end{enumerate}
\end{theorem}

Thus, for monotone classes of relational structures, the picture is clear.  Beyond monotone classes, not every NIP class is stable or monadically NIP.  However, it has been conjectured~\cite{tww4} that for any hereditary class $\C$ of structures, the model-checking problem on $\C$ is fixed-parameter tractable if, and only if, $\C$ is NIP.  This has previously been established for monotone classes of graphs (by the results of Adler and Adler, combined with those of Grohe et al.) and for hereditary classes of ordered graphs by results of Simon and Torunczyk~\cite{tww}.  Our results also extend the classes for which this conjecture is verified to all monotone classes of relational structures.

We establish some necessary definitions and notation in Sections~\ref{sec:prelim} and \ref{sec:path}.  The proof of Theorem~\ref{th:maintheorem} occupies the next three sections.  The equivalence of the first four  notions for any monotone class $\C$ is due to Braunfield and Laskowski~\cite{exmNIP}.  We therefore establish the equivalence of the first with the fifth and sixth.  In Section~\ref{sec:ip} we show that if $\Gaif(\C)$ is not nowhere dense, then $\C$ admits a formula with the independence property.   The converse is established in Section~\ref{sec:dense}.  That nowhere density of $\Gaif(\C)$ implies tractability is implicit in~\cite{nwdtract}.  We establish the converse of this statement in Section~\ref{sec:tractable}.

\section{Preliminaries}\label{sec:prelim}

    We assume familiarity with first-order logic and the basic concepts of model theory. We have tried to make this paper as self-contained as possible, but refer the reader to \cite{hodges} for background and undefined notation. Throughout this paper, $\L$ denotes a finite, first-order, relational language. We write $\mathrm{ar}(R)$ for the arity of each relation symbol $R \in \L$. Tuples of elements or variables are denoted by overlined letters and given a tuple $\bar a$ and $k\leq|\bar a|$, we write $\bar a(k)$ to denote the $k$-th element of $\bar a$. Often we abuse notation and treat tuples as unordered sets; whether we refer to the ordered tuple or the unordered set should be clear from the context. 

    We adopt the convention of allowing finitely many constant symbols (i.e.\ parameters) in $\L$-formulas. Syntactically, these are to be understood as additional free variables, while semantically these have a fixed interpretation in every $\L$-structure. This is purely a notational convenience and has no effect on the applicability of our results. By a further abuse of notation, we do not distinguish between a parameter $p$ and its interpretation $p^M$ in an $\L$-structure, $M$.
    
\subsection{Graphs and relational structures}

An $\L$-structure is denoted by $(M,R^M)_{R \in \L}$, where $M$ is its underlying set and $R^M \subseteq M^{\ar(R)}$ is the interpretation of the relation symbol $R \in \L$ in $M$. We write $\Cfrak(\L)$ for the class of all $\L$-structures. By abusing notation, often we do not distinguish between an $\L$-structure and its underlying set. For an $\L$-structure $M$ and a subset $A\subseteq M$ we denote by $M[A]$ the substructure of $M$ induced by $A$. By the \emph{equality type} of a tuple $\bar m$ from an $\L$-structure $M$, we mean the set $\Delta_=(\bar m)$ of atomic formulas $\eta(\bar x)$ using only the equality symbol such that $M \models \eta(\bar m)$.

A \emph{homomorphism} from an $\L$-structure $M$ to an $\L$-structure $N$ is a map $f:M \to N$ satisfying such that for all relation symbols $R \in \L$ and tuples $\bar m \in M^{\ar(R)}$, if $\bar m \in R^M$ then $f(\bar m) \in R^N$. A homomorphism of pointed structures $f:(M,\bar m) \to (N,\bar n)$ is understood as a homomorphism $f:M \to N$ of the underlying $\L$-structures such that $f(\bar m)=\bar n$. 

By a graph $G$ we mean an $\{E\}$-structure such that $E^G\subseteq G^2$ is a symmetric, irreflexive binary relation. We write $E(G)$ rather than $E^G$ for the edge set of a graph. Given a graph $G$ and $r \in \N$, we write $G^r$ for the \emph{$r$-subdivision of $G$}, i.e.\ the graph obtained by replacing each edge of $G$ by a path of length $r + 1$. We denote by $K_n$ the complete graph on $n$ vertices and by $K_{t,t}$ the complete bipartite graph with parts of size $t$. We write $G=(U,V;E)$ for a bipartite graph with parts $U$ and $V$ and edge set $E \subseteq U\times V$, and write $\Bfrak$ for the class of all bipartite graphs.

For a graph $G$, a vertex $u \in G$, and $m \in \N$ we write $N^G_m(u)$ for the \emph{$m$-neighbourhood of $G$ around $u$}, i.e.\ the vertices that are reachable from $u$ by a path of length at most $m$. For $A\subseteq G$, we write $N^G_m(A)$ for $\bigcup_{u \in A} N^G_m(u)$. Moreover, we write $N^G(u)$ for $\bigcup_{m \in \N} N^G_m(u)$, and similarly define $N^G(A)$.

 \begin{definition}[Gaifman graph/class]
Given an $\L$-structure $(M,R^M)_{M \in \L}$ we define the \emph{Gaifman graph} (or \emph{underlying graph}) of $M$, denoted $\Gaif(M)$, to be the graph on vertex set $M$ with edges:
\begin{center}
    $E := \{ (x,y):\exists R \in \L \exists v_1,\dots,v_{\mathrm{ar}(R)-2} \exists  \sigma \in S_{\mathrm{ar}(R)}(\sigma(x,y,v_1,\dots,v_{\mathrm{ar}(i)-2})\in R^M)\}, $
\end{center}
where $\Scal_n$ the symmetric group on $n$ elements. For a class of relational structures $\mathcal{C}$, all in the same language, we define the \emph{Gaifman class} of $\C$ to be $\Gaif(\C):=\{ \Gaif(M) : M \in \C \}$.
\end{definition}

\subsection{Sparsity and stability}\label{subsec:sparsity-stability}

Throughout this paper, $\C$ refers to a class of $\L$-structures or graphs. We write $\Th(\C)$ for the common theory of the class, i.e.\ the set of all first-order $\L$-sentences that hold in all structures in $\C$. We say that a class $\C$ is:
\begin{itemize}
       \item \emph{hereditary}, if $\C$ is closed under induced substructures;
       \item \emph{monotone}, if $\C$ is closed under weak substructures, i.e.\ if $(M, R^M)_{R \in \L} \in \C$ then $(M', R^{M'})_{R \in \L} \in \C$ for any $M' \subseteq M$ and $R^{M'}\subseteq R^M$.
 \end{itemize}

 \begin{definition}
    Let $\C$ be a class of graphs. We say that $\C$ is \emph{nowhere dense} if for every $r\in\N$ there is some $n\in\N$ such that for all $G\in\C$ we have that $K_n^r$ is not a subgraph of $G$. 
\end{definition}

 Nowhere density was introduced by Nešetřil and Ossona de Mendez \cite{nowhere}, as a structural property of classes of finite graphs that generalises numerous well-behaved classes, including graphs of bounded degree, planar graphs, graphs excluding a fixed minor and graphs of bounded expansion. Nowhere dense classes play an important role in algorithmic graph theory, as several computationally hard problems become tractable when restricted to such classes. 

Let us now recall some core notions of tameness from classification theory, adapted from the context of infinite structures to that of classes of (not necessarily infinite) structures.

\begin{definition}[Order/Independence Property]
    Let $\mathcal{C}$ be a class of $\L$-structures. We say that an $\L$-formula $\phi(\bar x,\bar y)$ has:
\begin{enumerate}
    \item \emph{The Order Property in $\mathcal{C}$} if for all $n\in \N$ there is some $M_n\in\mathcal{C}$ and sequences $(\bar a_i)_{i \in [n]}$  and $(\bar b_j)_{j \in [n]}$ of tuples from $M_n$ such that:
    \begin{center}
        $M_n\vDash \phi(\bar a_i,\bar b_j) \text{ if, and only if, } i<j.$
    \end{center}
    \item \emph{The Independence Property in $\mathcal{C}$} if for all $G=(U,V;E)\in\Bfrak$ there is some $M_G\in\mathcal{C}$ and sequences of tuples $(\bar a_i)_{i \in U}$ and $(\bar b_j)_{j\in V}$ such that:
    \begin{center}
        $M_n\vDash \phi(\bar a_i,\bar b_j) \text{ if, and only if, } (i,j)\in E.$
    \end{center}
\end{enumerate}
We say that $\mathcal{C}$ is \emph{stable} if no formula has the order property in $\mathcal{C}$. We say that $\mathcal{C}$ is \emph{NIP} (No Independence Property) if no formula has the independence property in $\mathcal{C}$. 
\end{definition}

An easy application of compactness reveals that a class $\C$ is stable (resp.\ NIP) if, and only if, all completions of $\Th(\C)$ are stable (resp.\ NIP) in the standard model-theoretic sense (see for instance \cite{She78} for the standard model-theoretic definitions). 

Given a class $\C$ of $\L$-structures and an expansion $\L'=\L\cup\{P_i:i\in I\}$ by unary predicates, we say that a class $\C'$ of $\L'$-structures is a \emph{monadic expansion} of $\C$ if $\C = \{ M'\restr_\L : M' \in \C' \}$, where for an $\L'$-structure $M'$ we write $M'\restr_\L$ for the $\L$-structure obtained from $M'$ by simply forgetting each relation symbol not in $\L$. In other words, $\C'$ is a monadic expansion of $\C$ if, for each structure $M\in\C$, $\C'$ contains at least one copy of $M$ expanded with unary predicates which are interpreted freely, and no other structures.

\begin{definition}[Monadic Stability/NIP]
    Let $\C$ be a class of $\L$-structures. We say that $\C$ is \emph{monadically stable} (resp.\ \emph{monadically NIP}) if all monadic expansions $\C'$ of $\C$ are stable (resp.\ NIP).
\end{definition}

The relationship between sparsity and stability is captured by the following theorem, which was established by Podewski and Ziegler \cite{podewskiziegler}, in the context of infinite graphs, and much later translated to the context of graph classes by Adler and Adler \cite{AdlerAdler2014}. 

\begin{theorem}[Adler, Adler \cite{AdlerAdler2014}; Podewski, Ziegler \cite{podewskiziegler}]\label{th:adleradler}
        Let $\C$ be a nowhere dense class of graphs. Then $\C$ is monadically stable. Moreover, the following are equivalent when $\C$ is monotone:
        \begin{enumerate}
            \item $\C$ is NIP;
            \item $\C$ is monadically NIP;
            \item $\C$ is stable;
            \item $\C$ is monadically stable;
            \item $\C$ is nowhere dense.
        \end{enumerate}
\end{theorem}

Furthermore, Adler and Adler asked if Theorem \ref{th:adleradler} can be generalised to arbitrary relational structures with finite signature. Recently, Braunfeld and Laskowski established a collapsing phenomeon akin to Theorem \ref{th:adleradler} for relational structures. 
    
\begin{theorem}[Braunfeld, Laskowski, \cite{exmNIP}]\label{th:exmNIP}
Let $\C$ be a hereditary class of structures. Then $\C$ is monadically NIP (resp.\ monadically stable) if, and only if, $\C$ is NIP (resp.\ stable). Moreover, if $\C$ is monotone then $\C$ is NIP if, and only if, it is stable. 
\end{theorem}

In light of the above, Theorem \ref{th:maintheorem} answers the question of Adler and Adler affirmatively by connecting the picture arising in Theorem \ref{th:exmNIP} with the sparsity-theoretic properties of the Gaifman class. 

\subsection{Model-checking}

    By model-checking on a class $\C$ we refer to the following parametrised decision problem:
    
    \begin{table}[!h]
        \begin{tabular}{rl}
             \textsc{Given}: & A $\FO$-sentence $\phi$ and a structure $M \in \C$. \\
             \textsc{Parameter}: &  $|\phi|$. \\
             \textsc{Decide}: &  Whether or not $M$ satisfies $\phi$.
        \end{tabular}
    \end{table}
    
    \begin{definition}
        We say that $\C$ is \emph{tractable}, or that the model-checking problem on a class $\C$ is fixed-parameter tractable, if there is an algorithm that decides on input $(M, \phi)$ whether $G \models \phi$, in time $f(|\phi|) \cdot |M|^{\mathcal{O}(1)}$ for some computable function $f$. 
    \end{definition}

Model-checking on the class of all graphs is complete with respect to the complexity class $\mathsf{AW}[*]$, which is conjectured to strictly contain the class $\FPT$. We shall assume throughout that $\mathsf{AW}[*]\neq \FPT$. 

All hereditary classes of graphs and relational structures that are known to admit tractable model-checking are NIP. Moreover, the robustness of NIP in hereditary classes hints at its potential necessity for tractability. This is the basis of the following conjecture:
    
\begin{conjecture}[\cite{tww4}]\label{conj}
    Let $\C$ be a hereditary class of relational structures. Then $\C$ is tractable if, and only if, $\C$ is NIP.
\end{conjecture}

There is good evidence for a positive answer to this conjecture. Indeed, it is known to hold for:
    \begin{itemize}
        \item Monotone classes of graphs, where NIP coincides with nowhere density \cite{nwdtract};
        \item Hereditary classes of ordered graphs, where NIP coincides with bounded twin-width \cite{tww}.
    \end{itemize}
    
Although it is not explicitly stated in this form, a careful examination of the argument of \cite{nwdtract} reveals that the following holds. 
    
\begin{theorem}[Grohe, Kreutzer, Siebertz, \cite{nwdtract}]
    Let $\C$ be a class of relational structures such that $\Gaif(\C)$ is nowhere dense. Then $\C$ admits fixed-parameter tractable model-checking. 
\end{theorem}

\subsection{Interpretations}

Interpretations in classical model theory allow us to find structures in some language in a definable way inside a definable quotient of structures in some other language, mimicking, for instance, the way one can find the rational numbers inside the integers. 

In our case, we focus on a restricted version of interpretations, which we call \emph{simple interpretations} (possibly with parameters). Intuitively, a class of $\L^\prime$-structures $\Dcal$ can be interpreted in a class of $\L$-structures $\C$ if there is a uniform way of defining every structure in $\Dcal$, in some (Cartesian power of some) structure $\C$. More formally:

\begin{definition}[Simple interpretation]\label{def:simple-interpretation}
Let $\L,\L'$ be two finite relational languages. A \emph{simple interpretation with parameters} $I:\Cfrak(\L) \to \Cfrak(\L')$ consists of the following data:
\begin{itemize}
    \item A \emph{domain formula} $\delta(\bar x,\bar v)\in\L$ and, a function $d$ which to each $M\in\Cfrak(\L)$ associates a tuple $\bar d(M)$ from $M^{|\bar v|}$.
    \item For each $k$-ary relation symbol $R(y_1,\dots,y_k)\in\L'$ an \emph{interpreting formula} $\phi_R(x_1,\dots,x_k,\bar v_R)\in\L$, where $|\bar x_i| = |\bar x|$, for each $i\in[k]$, and a function $c_R$ which to each $M\in\Cfrak(\L)$ associates a tuple $\bar c_R(M)$ from $M^{|\bar v_R|}$.
\end{itemize}
\end{definition}

In order to make our discussion of interpretations easier, we adopt the following notation. Given $M \in \Cfrak(\L)$ we write $I(M)$ for the $\L'$ structure on the set $\delta(M) := \{a\in M:M\vDash\delta(a,\bar d(M))\}$ with:
\begin{center}
    $I(M)\vDash R(a_1,\dots,a_k) \text{ if, and only if, } M\vDash\phi_R(a_1,\dots,a_k,\bar c_R(M)),$
\end{center}
for each $k$-ary relation symbol $R\in\L'$ and $a_1,\dots,a_k\in\delta(M,\bar{d}(M))$. This dually gives a map $\widehat{I}:\L' \to \L$ mapping $\L'$-formulas to $\L$-formulas with parameters, such that for any $\L'$-sentence $\phi$ we have that:
\begin{center}
    $M \models \widehat{I}(\phi) \text{ if, and only if, } I(M) \models \phi.$
\end{center}

In order to be able to reduce the problem of $\FO$ model-checking from one class of structures to another, possibly in a different language, we are interested in interpretations that can be computed in polynomial time. More precisely we define the following notion:

\begin{definition}[Polynomial interpretation]
    Given classes of structures $\C \subseteq \Cfrak(\L)$ and $\Dcal \subseteq \Cfrak(\L')$ we say that $\Dcal$ is \emph{polynomially interpreted in $\C$, with parameters,} if there are:
    \begin{enumerate}
        \item A simple interpretation with parameters, $I:\Cfrak(\L) \to \Cfrak(\L')$, as in \Cref{def:simple-interpretation}, such that the functions $d$ and $(c_R)_{R\in\L^\prime}$ are computable in polynomial time; and
        \item a polynomial-time computable map $f:\Dcal \to \C$ such that for all $D\in\Dcal$ we have that $D = I(f(D))$. 
    \end{enumerate}
    In this case, we write $\Dcal\leq_\Psf\C$.
\end{definition}

The next lemma justifies why polynomial interpretations are useful.

\begin{lemma}\label{lem:simple-interpretations-params}
    The relation $\leq_\Psf$ is a quasi-order on the collection of classes of structures in finite relational languages. Moreover $\leq_\Psf$ preserves tractability, i.e.\ if $\C$ is tractable and $\Dcal \preceq_\Psf \C$, then $\Dcal$ is tractable. 
\end{lemma}

\begin{proof}
The first part of the lemma is immediate, so let us only discuss the second part. We reduce the problem of model checking in $\Dcal$ to model checking in $\C$. Given an $\L^\prime$-sentence $\phi$ and an $\L^\prime$-structure $M\in\Cfrak(\L^\prime)$, we can compute, by assumption, in polynomial time an $\L$-structure $f(D)\in\C$ such that $M=I(f(D))$. By assumption, we can also compute $I(f(D))$ in polynomial time, since the parameters in the domain and interpreting formulas are computable from $M$ in polynomial time. Then, we have that:   
\begin{center}
    $f(D) \models \widehat{I}(\phi) \text{ if, and only if, } I(f(D)) = M \models \phi,$
\end{center} 
where $\widehat{I}(\phi)$ is obtained, essentially, as in the discussion after \Cref{def:simple-interpretation}, which can clearly be done in polynomial time, from $\phi$. Since $\C$ is tractable, it follows that $\Dcal$ is tractable.
\end{proof}

\subsection{Ramsey Theory}\label{subs:ramsey}

A core technique that is used repeatedly in our arguments is that if a finite structure is large enough, then patterns in it are inevitable. This is the main idea of Ramsey theory, the relevant tools from which we recall here.

\begin{theorem}[Ramsey's Theorem, \cite{ramsey}]\label{th:ramsey}
    There is a computable function $\Rcal:\N^3 \to \N$ such that for all $m,k,r \in \N$ and for every colouring $\chi: [\Rcal(m,k,r)]^{(k)} \to [r]$ there exists some $S \subseteq [\Rcal(m,k,r)]$ of size $m$ which is monochromatic. 
\end{theorem}

Another standard theorem from Ramsey theory that we make use of is the following well-known variant of \Cref{th:ramsey}:

\begin{theorem}[Bipartite Ramsey Theorem]\label{th:partramsey}
    There is a computable function $\Pcal:\N^2 \to \N$ such that for all $m,r\in\N$ and all edge colourings of the complete bipartite graph  $K_{\Pcal(m,r),\Pcal(m,r)}$ with $r$ colours, there are subsets $A_,B$ of the two parts, both of size $m$, which induce a monochromatic copy of $K_{m,m}$.
\end{theorem}

We also need to make use of the following Ramsey-theoretic result, where the number of colours is allowed to be possibly infinite. Of course, in this case, we cannot expect to find monochromatic subsets. Nonetheless, we can ensure that the behaviour of the colouring falls into one of few ``canonical'' cases on a large enough set. The original canonical Ramsey theorem is due to Erd\H{o}s and Rado \cite{erdos}, but for the purposes of this paper, we are only interested in the bipartite version in its effective form. 

\begin{theorem}[Bipartite Canonical Ramsey Theorem,  \cite{KMV_2017}]\label{canramsey}
There is a computable function $\Kcal: \N \to \N$ such that for every $n \in \N$ and every edge-colouring of the complete bipartite graph $K_{\Kcal(n),\Kcal(n)}$ there exist subsets $X$,$Y$ of the two parts, both of size $n$, such that one of the following occurs for all $x, x' \in X$ and $y, y' \in Y$:
\begin{enumerate}
    \item $\chi(x,y)=\chi(x',y')$;
    \item $\chi(x,y)=\chi(x',y')$ if, and only if, $x = x'$;
    \item $\chi(x,y)=\chi(x',y')$ if, and only if, $y=y'$;
    \item $\chi(x,y)=\chi(x',y')$ if, and only if, $x = x'$ and $y = y'$.
\end{enumerate}

    \begin{figure}[H]
  \centering\small

  $(1):$
  \adjustbox{scale = 0.7}{
\begin{tikzcd}
\CIRCLE \arrow[rrdd, no head, color = blue, line width = 1.5pt] \arrow[rr, no head, color = blue, line width = 1.5pt] \arrow[rrdddd, no head, color = blue, line width = 1.5pt] &  & \CIRCLE \\
                                                                    &  &   \\
\CIRCLE \arrow[rruu, no head, color = blue, line width = 1.5pt] \arrow[rr, no head, color = blue, line width = 1.5pt] \arrow[rrdd, no head, color = blue, line width = 1.5pt]   &  & \CIRCLE \\
                                                                    &  &   \\
\CIRCLE \arrow[rr, no head, color = blue, line width = 1.5pt] \arrow[rruu, no head, color = blue, line width = 1.5pt] \arrow[rruuuu, no head, color = blue, line width = 1.5pt] &  & \CIRCLE
\end{tikzcd}
}
  $(2):$
  \adjustbox{scale = 0.7}{
\begin{tikzcd}
\CIRCLE \arrow[rrdd, no head, color = blue, line width = 1.5pt] \arrow[rr, no head, color = blue, line width = 1.5pt] \arrow[rrdddd, no head, color = blue, line width = 1.5pt] &  & \CIRCLE \\
                                                                    &  &   \\
\CIRCLE \arrow[rruu, no head, color = red, line width = 1.5pt] \arrow[rr, no head, color = red, line width = 1.5pt] \arrow[rrdd, no head, color = red, line width = 1.5pt]   &  & \CIRCLE \\
                                                                    &  &   \\
\CIRCLE \arrow[rr, no head, color = green, line width = 1.5pt] \arrow[rruu, no head, color = green, line width = 1.5pt] \arrow[rruuuu, no head, color = green, line width = 1.5pt] &  & \CIRCLE
\end{tikzcd}
}
  $(3):$
  \adjustbox{scale = 0.7}{
\begin{tikzcd}
\CIRCLE \arrow[rrdd, no head, color = red, line width = 1.5pt] \arrow[rr, no head, color = blue, line width = 1.5pt] \arrow[rrdddd, no head, color = green, line width = 1.5pt] &  & \CIRCLE\\
                                                                    &  &   \\
\CIRCLE\arrow[rruu, no head, color = blue, line width = 1.5pt] \arrow[rr, no head, color = red, line width = 1.5pt] \arrow[rrdd, no head, color = green, line width = 1.5pt]   &  & \CIRCLE \\
                                                                    &  &   \\
\CIRCLE\arrow[rr, no head, color = green, line width = 1.5pt] \arrow[rruu, no head, color = red, line width = 1.5pt] \arrow[rruuuu, no head, color = blue, line width = 1.5pt] &  & \CIRCLE
\end{tikzcd}
}
  $(4):$
  \adjustbox{scale = 0.7}{
\begin{tikzcd}
\CIRCLE \arrow[rrdd, no head, color = blue, line width = 1.5pt] \arrow[rr, no head, color = red, line width = 1.5pt] \arrow[rrdddd, no head, color = black, line width = 1.5pt] &  & \CIRCLE \\
                                                                    &  &   \\
\CIRCLE \arrow[rruu, no head, color = yellow, line width = 1.5pt] \arrow[rr, no head, color = green, line width = 1.5pt] \arrow[rrdd, no head, color = brown, line width = 1.5pt]   &  & \CIRCLE \\
                                                                    &  &   \\
\CIRCLE \arrow[rr, no head, color = purple, line width = 1.5pt] \arrow[rruu, no head, color = orange, line width = 1.5pt] \arrow[rruuuu, no head, color = pink, line width = 1.5pt] & &  \CIRCLE
\end{tikzcd}
}

\end{figure} 
\end{theorem}

Henceforth, we shall say that an edge colouring of a complete bipartite graph is \emph{canonical of type} $1$ (resp.\ $2,3,4$) if it satisfies condition 1 (resp.\ $2,3,4$) from Theorem \ref{canramsey} for all edges. More generally, we say that such a colouring is \emph{canonical} whenever it is canonical of any type.

\subsection{Canonical structures}
Recall that a formula $\phi(\bar x)$ is called \emph{primitive positive} if it has the form $\exists \bar y \psi(\bar x, \bar y)$, where $\psi$ is a conjunction of atomic formulas. We say that $\phi(\bar x)$ is \emph{quasi-positive} if it has the form $\exists \bar y \psi(\bar x, \bar y)$, where $\psi$ is a conjunction of atomic formulas and formulas of the form $v_i\neq v_j$, where $v_i,v_j$ are variables of $\phi$.  Primitive positive formulas are also known as \emph{conjunctive queries} in the database theory literature.  The following association of a canonical structure with a primitive positive formula and conversely a canonical such formula with a finite structure goes back to Chandra and Merlin~\cite{CM77}.

\begin{definition}[Canonical structures]
Given a primitive positive formula $\phi(\bar x) = \exists \bar y \psi(\bar x, \bar y)$ we define a pointed $\L$-structure $(M_\phi, \bar x)$ whose domain is the set $\{v_1,\dots,v_r\}$ of variables of $\phi$, and where each $R \in \L$ is interpreted as follows:
\[M_\phi\vDash R(v_1, \dots,v_n) \iff R(v_1,\dots,v_n) \text{ appears as a conjunct in $\psi(\bar x,\bar y)$.}\]

The pointed elements $\bar x$ precisely correspond to the free variables of $\phi$. This structure is unique, up to isomorphism, and we call it \emph{the canonical structure of $\phi$}.
\end{definition}

Similarly, for every pointed $\L$-structure $(A,\bar x)$ we may associate a primitive positive formula $\phi_A(\bar x)$ so that $(\Mcal_{\phi_A},\bar x)=(A,\bar x)$. We call this formula the \emph{canonical formula} of $(A,\bar x)$. Let $\phi(\bar x)$ be a primitive positive formula and $(M_\phi,\bar x)$ its canonical structure. It is easy to see that for any $\L$-structure $A$ and $\bar a \in A$ we have that $A \models \phi(\bar a)$ if, and only if, there exists a homomorphism (of pointed structures) $h: (\Mcal_\phi,\bar x) \to (A, \bar a)$. 

In fact, the satisfaction of a primitive positive formula depends only on the satisfaction of each individual connected component of its canonical structure. More precisely we have the following:

\begin{lemma}\label{ppcomponents}
Let $\phi(\bar x)$ be a primitive positive formula, and $\Mcal_\phi$ its canonical structure. Let $G_1,\dots G_n$ be the connected components of $\Gaif(\Mcal_\phi)\setminus \bar x$, and $M_i = \Mcal_\phi[G_i,\bar x]$. Then, the following are equivalent, for any $\L$-structure $A$ and tuple $\bar a \in A$:
\begin{enumerate}
    \item $A \models \phi(\bar a)$.
    \item  For each  $i \in [n]$ there exists a homomorphism  $h_i:(M_i,\bar x)\to(A,\bar a)$.
\end{enumerate}
\end{lemma}

\begin{proof}
Only the implication $(2)\implies(1)$ is non-trivial. Given a collection of homomorphisms $h_i:(M_i,\bar x)\to (A,\bar a)$ as in $(2)$, their union $h = \bigcup h_i$ is a homomorphism $(\Mcal_\phi,\bar x)\to (A,\bar a)$. This is well-defined as the components $G_1,\dots, G_n$ are disjoint, and hence $M_i\cap M_j=\bar x$ for $i\neq j$ which has the same image under all the $h_i$'s.  Also, if $(v_1, \dots, v_n) \in R^{\Mcal_\phi}$ then the elements in $\bar v \setminus \bar x$ lie in the same connected component of $\Gaif(\Mcal_\phi)\setminus \bar x$, say $G_k$. Since $h_k$ is a homomorphism, it follows that $h_k(\bar v)=h(\bar v) \in R^A$. 
\end{proof}

Given a quasi-positive formula $\phi(\bar x)=\exists \bar y \psi(\bar x,\bar y)$, we write $\tilde \phi(\bar x)$ for the primitive positive formula obtained by removing the conjuncts of the form $v_i \neq v_j$ from $\psi$. We then define the canonical structure of $\phi$ to be $\Mcal_{\tilde \phi}$. It is easy to see that, in this case, $A \models \phi(\bar a)$ if, and only if, there is a homomorphism $h: (\Mcal_{\tilde \phi},\bar x) \to (A, \bar a)$ satisfying $f(v_i)\neq f(v_j)$ for every conjunct $v_i \neq v_j$ from $\psi$. 

\section{Path formulas}\label{sec:path}
In our analysis, we argue that whenever a monotone class  
of relational structures has the independence property then this is witnessed by a certain kind of primitive positive formula. In the case of graphs, it is implicit in the work of Adler and Adler that the canonical structure of this primitive positive formula is a path in the standard graph-theoretic sense, i.e.\ a tuple $(x_1,\dots,x_n)$ of pairwise distinct elements such that $E(x_i,x_{i+1})$ for all $i \in [n-1]$. 

In this section, we introduce the analogue of (graph) paths that witnesses the independence property in general relational structures. We start with the following rather technical definition. 

\begin{definition}[Path]\label{def:path}
By a \emph{path of length $n$}, we mean an $\L$-structure $\mathbf{P}$ consisting of a sequence of tuples $\bar e_1, \dots, \bar e_n$ of non-repeating elements such that: 
\begin{itemize}
    \item $\mathbf{P}=\bigcup_{i \in [n]}\bar e_i$;
    \item $|\bar e_i\cap \bar e_{i+1}|=1$, for all $i < n$;
    \item $\bar e_i \not\subseteq \bar e_{i+1}$ and $\bar e_{i+1}\not\subseteq \bar e_i$, for all $i<n$;
    \item $\bar e_i\cap \bar e_{j} = \emptyset$, for all $j \in [n]\setminus\{i-1,i,i+1\}$;
    \item $R_i(\bar e_i)$, for some unique relation symbol $R_i \in \L$;
    \item $R(\bar a) \implies \bar a = \bar e_i$ for some $i \in [n]$, for all relation symbols $R \in \L$ and all tuples $\bar a \in \mathbf{P}$. 
\end{itemize}

We write $S(\textbf{P})= \bar e_1\setminus \bar e_2$ and call these the \emph{starting vertices}, while we write $F(\PP)= \bar e_n \setminus \bar e_{n-1}$ and call these the \emph{finishing vertices}. We refer to the tuples $\bar e_i$ as the steps of the path, and to the singletons in $\bar e_i \cap \bar e_{i+1}$ as the joints of the path. 
\end{definition}

Given a primitive positive formula $\phi(\bar x,\bar y,\bar z)$ (where $\bar z$ is possibly empty), we say that $\phi$ is a \emph{path formula} if there are $x_0 \in \bar x$ and $y_0 \in \bar y$ such that $\Mcal_\phi$ is a path with $x_0 \in S(\Mcal_\phi)$ and $y_0 \in F(\Mcal_\phi)$. Similarly, we call $\phi$ a \emph{simple path formula} if $\bar x \subseteq S(\Mcal_\phi)$ and $\bar y \subseteq F(\Mcal_\phi)$.

Note that technically, no graph $G$ can be a path under the above definition. Indeed, the last condition ensures that $E(G)$ cannot be symmetric as no permutation of a tuple appearing in a relation $R$ can appear in any other relation from $\L$. To avoid confusion, we always refer to paths in the standard graph-theoretic sense as \emph{graph paths}. 

Intuitively, a path formula $\phi(\bar x,\bar y)$ plays the role of a higher arity graph path from $\bar x$ to $\bar y$. However, under enough symmetry, it is possible that we cannot definably tell the direction of $\phi$, i.e. $\bar x$ and $\bar y$ look the same within $\phi$. This is formalised in the following definition, and is important in the proof of Theorem \ref{thm:intractability}.

\begin{definition}[Symmetric path]\label{def:symmetric}
A \emph{symmetric path} is a path $\mathbf{P}$ of length $n$, such that $R_i=R_{n+1-i}$ for all $i \in [n]$. A \emph{symmetric path formula} $\phi(\bar x,\bar y,\bar z)$ is a simple path formula with $|\bar x|=|\bar y|=m$ such that $\Mcal_\phi$ is a symmetric path and there is an automorphism $f$ of $\Mcal_\phi$ which maps
    $\bar x=(x_1,\dots,x_m) \mapsto (y_{\sigma(1)},\dots,y_{\sigma(m)})$ and
    $\bar y=(y_1,\dots,y_m) \mapsto (x_{\sigma^{-1}(1)},\dots,x_{\sigma^{-1}(m)})$,
for some $\sigma \in \Scal_m$ which is not the identity permutation. Moreover, if $\phi$ contains parameters then these must be fixed by $f$.
\end{definition}

Given an $\L$-structure and a graph path in $\Gaif(M)$, we may produce a path formula that describes a  ``type'' for this path. This idea is captured by the following definition which is relevant for the proof of Lemma \ref{lem:bipartite}. 

\begin{definition}[Path type]\label{def:pathtype}
    Let $M$ be an $\L$-structure, and $S=(u_1,\dots,u_n)$ a graph path in $\Gaif(M)$. For every $i \in [n-1]$ we may associate a relation symbol $R_i \in \L$, elements $v_{i,1},\dots,v_{i,\ar(R_i)}$, and a permutation $\sigma_i \in S_{\ar(R_i)}$ such that $\sigma_i(u_i,u_{i+1},\bar v_i) \in R_i^M$. Then we call the formula 
    \[
    \begin{aligned}
        \phi(x,y,z_2,\dots,z_{n-1})=\exists \bar v_i \dots \bar v_{n-1}( R_1(\sigma_1(x,z_2,\bar v_1))\land R_2(\sigma_2(z_2,z_3,\bar v_2)) \\ \land \dots \land R_{n-1}(\sigma_{n-1}(z_{n-1},y,\bar v_{n-1})))
    \end{aligned}
    \]
    a \emph{path type} for the graph path $u_1,\dots, u_n$.  
\end{definition}

It is easy to see that whenever $S=(u_1,\dots,u_n)$ is a graph path in $\Gaif(M)$, then there is a path type $\phi$ for $S$ such that $M \models \phi(u_1,u_n,u_2,\dots, u_{n_1})$. Clearly, this is not uniquely determined by $S$, as for the same graph path $u_1,\dots,u_n$ in $\Gaif(M)$ we can possibly obtain different sequences of relations $R_i,\dots R_{i-1}$ and permutations $\sigma_1,\dots,\sigma_{i-1}$ as in Definition \ref{def:pathtype}. 

\section{From somewhere density to IP}\label{sec:ip}

The main result in this section is \Cref{prop:paths}, where we prove that for any monotone class $\C$ of relational structures whose Gaifman class is somewhere dense, there is a path formula (in the sense of \Cref{def:path}) which codes the edge relation of all bipartite graphs uniformly over $\C$. Our tools in the proof are the Ramsey-theoretic results from Section \ref{subs:ramsey}. These allow us to ensure various kinds of regularity, provided we work with suitably large finite structures. 

We work towards this theorem via two preparatory lemmas, which have the benefit of applying to classes that are not necessarily monotone. Intuitively, Lemma \ref{lem:bipartite} tells us that if $\C$ is a monotone class of relational structure whose Gaifman class is somewhere dense, then we can find a path formula that codes the edge relation of all finite complete bipartite graphs in $\C$.

\begin{lemma}\label{lem:bipartite}
    Let $\C$ be a class of $\L$-structures such that $\Gaif(\C)$ is somewhere dense. Then there is a path formula $\phi(x,y,\bar z)=\exists \bar w \psi(x,y,\bar z, \bar w)$ of length $\geq 2$ whose joints are precisely the variables in $\bar z$, and for each $n \in \N$ there is some $M_n \in \C$ and pairwise distinct elements $(a_i)_{i \in [n]}, (b_j)_{j \in [n]}, (\bar c_{i,j})_{(i,j) \in [n]^2}$ from $M_n$ such that
    \[M_n \models \phi(a_i,b_j,\bar c_{i,j}), \text{ for all }i,j \in [n].\]
\end{lemma}

\begin{proof}
    If $\Gaif(\C)$ is somewhere dense, then there exists $r \in \N$ such that for all $n \in \N$ there is some $M_n \in \Gaif(\C)$ with $K_n^r \leq \Gaif(M_n)$. Without loss of generality, we may assume that $r\geq 1$. Indeed, if $r=0$ then $K^1_n \leq K_{n^2} \leq \Gaif(M_{n^2})$ so we may pass to a subsequence of $(M_n)_{n \in \N}$ and relabel the indices appropriately. 
    
    For every $i < j$ from $[n]$ let $S^n_{i,j}$ be the graph path in $\Gaif(M_n)$ corresponding to the $r$-subdivision of the edge $(i,j)$ from $K_n$, directed from $i$ to $j$. Let $q \in \N$ be the maximum arity of a relation symbol $R \in \L$. Observe that there are at most $p=(|\L|\times q!)^{r+1}$ path types for each graph path $S^n_{i,j}$. By Ramsey's theorem we may find for each $n$ some $\Sigma_n \subseteq [\Rcal(n,2,q)]$ of size $n$ such that $S^{\Rcal(n,2,q)}_{i,j}$ have the same path type for all $i<j$ from $\Sigma_n$. By passing to a subsequence of $(M_n)_{n \in \N}$ and relabelling indices, we may therefore assume that all the $S^n_{i,j}$ have the same path type. Let this be $\phi_n$. Since there are only finitely many possible path types for every $n$, we may prune the sequence $(M_n)_{n \in \N}$ once again to ensure  that the same path type $\phi(x, y,\bar z)$ is obtained for all $n \in \N$. By definition, the joints of $\Mcal_\phi$ are precisely the variables in $\bar z$, while $\Mcal_\phi$ has length $\geq 2$ since $r \geq 1$.

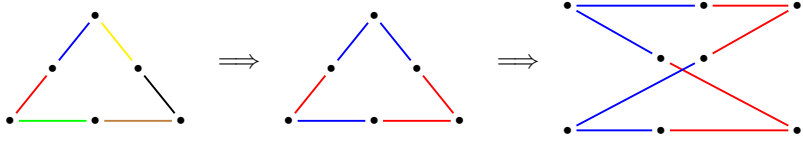
\begin{figure}[H]
  \centering\small

  \adjustbox{scale=0.4}{
\begin{tikzcd}
                                                 &                              & \CIRCLE                     &                              &                              \\
                                                 &                              &                             &                              &                              \\
                                                 & \CIRCLE \arrow[ruu, no head, color = blue, line width = 1.8pt] &                             & \CIRCLE \arrow[luu, no head, color = yellow, line width = 1.8pt] &                              \\
                                                 &                              &                             &                              &                              \\
\CIRCLE \arrow[rr, no head, color= green, line width = 1.8] \arrow[ruu, no head, color = red, line width = 1.8] &                              & \CIRCLE \arrow[rr, no head, color= brown, line width = 1.8] &                              & \CIRCLE \arrow[luu, no head, line width = 1.8]
\end{tikzcd}
}
  $\implies$
  \adjustbox{scale=0.4}{
\begin{tikzcd}
                                                 &                              & \CIRCLE                     &                              &                              \\
                                                 &                              &                             &                              &                              \\
                                                 & \CIRCLE \arrow[ruu, no head, color = blue, line width = 1.8pt] &                             & \CIRCLE \arrow[luu, no head, color = blue, line width = 1.8pt] &                              \\
                                                 &                              &                             &                              &                              \\
\CIRCLE \arrow[rr, no head, color= blue, line width = 1.8] \arrow[ruu, no head, color = red, line width = 1.8] &                              & \CIRCLE \arrow[rr, no head, color= red, line width = 1.8] &                              & \CIRCLE \arrow[luu, no head, color = red, line width = 1.8]
\end{tikzcd}
}
  $\implies$
  \adjustbox{scale=0.4}{
\begin{tikzcd}
\CIRCLE \arrow[rrrr, no head, line width = 1.8, color=blue]                                  &  &  &                                                         & \CIRCLE \arrow[rrr, no head, line width = 1.8, color=red]   &  &  & \CIRCLE \\
                                                      &  &  &                                                         &                       &  &  &         \\
                                                      &  &  & \CIRCLE \arrow[rrrrddd, no head, line width = 1.8, color = red] \arrow[llluu, no head, line width = 1.8, color = blue] & \CIRCLE \arrow[rrruu, no head, color = red, line width = 1.8] &  &  &         \\
                                                      &  &  &                                                         &                       &  &  &         \\
                                                      &  &  &                                                         &                       &  &  &         \\
\CIRCLE \arrow[rrr, no head, line width = 1.8, color = blue] \arrow[rrrruuu, no head, line width = 1.8, color = blue] &  &  & \CIRCLE \arrow[rrrr, no head, line width = 1.8, color = red]                           &                       &  &  & \CIRCLE
\end{tikzcd}
}

\end{figure}
    
    Work in $M_{2n}$ and let $(a_i)_{i \in [n]}$ be the elements corresponding to $1,\dots, n$ from $K^r_{2n}$, and $(b_j)_{j \in [n]}$ be those corresponding to $n+1,\dots, 2n$. Moreover, let $\bar c_{i,j}$ be the tuples obtained by removing $a_i$ and $b_j$ from the beginning and end respectively of the graph path $S^{2n}_{i,n+j}$. It is clear that the elements $(a_i)_{i \in [n]}, (b_j)_{j \in [n]}, (\bar c_{i,j})_{i,j \in [n]}$ are pairwise distinct. Since the path type of $S^{2n}_{i,n+j}$ is equal to $\phi$ for all $i,j \in [n]$, it follows that $ M_{2n} \models \phi(a_i,b_j,\bar c_{i,j})$, for all $i,j \in [n]$.  We finally pass to the subsequence $(M_{2n})_{n \in \N}$ and relabel.
\end{proof}

Having established that we may encode the edge relation of any complete bipartite graph, we want to use monotonicity in order to encode the edge relation of arbitrary bipartite graphs, and consequently, to witness the independence property. To achieve this, we must ensure that the tuples used in the encoding are ``sufficiently disjoint'' so that the removal of the desired relations does in fact translate to the removal of an encoded edge. The following lemma is a step toward this.

\begin{lemma}{disjoint}\label{lem:disjoint}
 Let $\C$ be a class of $\L$-structures such that $\Gaif(\C)$ is somewhere dense. Then there is a path formula $\phi(\bar x,\bar y,\bar z)=\exists \bar w \psi(\bar x,\bar y,\bar z,\bar w)$ of length $\geq 2$ with parameters $\bar p$ whose joints are precisely the elements of $\bar z$, and for every $n \in \N$ there is some $M_n \in \C$ and tuples $(\bar a_i)_{i \in [n]},$ $(\bar b_j)_{j \in [n]},$ $(\bar c_{i,j})_{(i,j) \in [n]^2},$ $(\bar d_{i,j})_{i,j \in [n]^2}$ from $M_n$ such that the following hold for all $i,i',j,j' \in [n]$: 
\begin{enumerate}
    \item $M_n \models \psi(\bar a_i,\bar b_j,\bar c_{i,j},\bar d_{i,j})$;
    \item $\bar a_i(k) \neq \bar a_{i'}(k)$, for $i\neq i'$ and all $k \in [|\bar x|]$;
    \item $\bar b_j(k) \neq \bar b_{j'}(k)$, for $j\neq j'$ and all $k \in [|\bar y|]$;
    \item $\bar c_{i,j}(k)\neq \bar c_{i',j'}(k)$ and $\bar c_{i,j}(k)\neq \bar c_{i,j}(l)$, for $(i,j)\neq (i',j')$ and all $k\neq l$ from $[|\bar z|]$;
    \item $\bar d_{i,j}(k)\neq \bar d_{i',j'}(k)$, for $(i,j)\neq (i',j')$ and all $k \in [|\bar w|]$.
\end{enumerate}
\end{lemma}

\begin{proof}
    Let $\phi(x,y,\bar z)=\exists \bar w \psi(x,y,\bar z,\bar w)$ and $(M_n)_{n \in \N}$ be as in \Cref{lem:bipartite}. For clarity, we write $(a^n_i)_{i \in [n]}, (b^n_j)_{j \in [n]}, (\bar c^n_{i,j})_{(i,j) \in [n]^2}$ to denote the elements of $M_n$ from the same lemma. For each $n \in \N$, and for each pair $(i,j)\in [n]^2$, pick a tuple $\bar d^n_{i,j}$ of elements from $M_n$ consisting of some arbitrarily fixed existential witnesses to $M_n \models \phi(a^n_i,b^n_j,\bar c^n_{i,j})$, i.e.\ $M_n \models \psi(a^n_i,b^n_j, \bar c^n_{i,j},\bar d^n_{i,j})$ for all $i,j \in [n]$.

Let $m=| \bar d_{i,j} |$. By $m$ applications of \Cref{canramsey}, we may assume that whether $\bar d^n_{i,j}(k)=\bar d^n_{i',j'}(k)$ depends on one of the four canonical cases from that theorem, and not on $n$. Indeed, for every $n \in \N$ and each $k \in [m]$, define colourings $\chi_{n,k}(i,j) = \bar d^n_{i,j}(k)$ of the edges of $K_{n,n}$.  Let $\Kcal:\N\to\N$ be the computable function guaranteed by \Cref{canramsey} and write $\Kcal^m$ for the composition of $\Kcal$ with itself $m$ times. It follows that the complete bipartite graph with parts of size $\Kcal^m(n)$ contains subsets $A_n,B_n$ of the two parts of size $n$, which induce a copy of $K_{n,n}$ on which $\chi_{\Kcal^m(n),k}$ is canonical for all $k \in [m]$. We may thus restrict the argument on the subsequence $(M_{\Kcal^m(n)})_{n \in \N}$ and the elements $a^{\Kcal^m(n)}_i, b^{\Kcal^m(n)}_j, \bar c^{\Kcal^m(n)}_{i,j}, \bar d^{\Kcal^m(n)}_{i,j}$ for $i \in A_n$ and $j \in B_n$ and relabel appropriately. For every $n \in \N$, after the relabelling, we have thus obtained a tuple $\bar t_n \in [4]^m$ such that $\chi_{n,k}$ is canonical of type $\bar t_n(k)$. Since there are only finitely many such $\bar t_n$, by the pigeonhole principle we may consider a subsequence of $(M_n)_{n \in \N}$ for which $\bar t_n$ is constant and equal to some $\bar t\in [4]^m$, and relabel once more. 

We now proceed to sequentially remove elements from the tuples $\bar d^n_{i,j}$, and to either name them by a parameter, or to append them to one of $a^n_i$ or $b^n_j$. Since $\bar t$ is constant for all $n$, exactly the same process is carried out to all tuples $\bar d_{i,j}$, and so we may concurrently move the corresponding variables from $\phi$. So, if we fall into Case $1$ for some $k$, i.e. if $\bar t(k)=1$, then $\bar d_{i,j}(k)$ is the same for all $i,j$, and so we may name it by a parameter and remove it from every $\bar h_{i,j}$. If we fall into Case $2$, then $\bar d_{i,j}(k)=\bar d_{i',j'}(k)$ if, and only if, $i = i'$. Then, for every $i \in [n]$ we may remove the common element $\bar d_{i,j}(k)$ from each $\bar d_{i,j}$ and append it to $a_i$, turning it into a tuple $\bar a_i$. We then adjust $\phi$ accordingly by shifting the corresponding variable $v_k$ from $\bar v$ to $x$, which also becomes a tuple $\bar x$. Case $3$ is symmetric to Case $2$, only now we append $\bar d_{i,j}(k)$ to $\bar b_j$ and shift the variable $v_k$ to $\bar y$. We may therefore assume that we fall into Case $4$ for all the remaining $k \in [m]$. 

We argue that the resulting formula and tuples satisfy the requirements of the lemma. Clearly, $M_n \models \phi(\bar a_i,\bar b_j,\bar c_{i,j},\bar d_{i,j})$ for all $n \in \N$ and $i,j \in [n]$. Condition $2$ is also satisfied, since the original singletons $(a_i)_{i \in [n]}$ were pairwise disjoint, while for every $i \neq i'$ and $k \in [m]$ the elements $\bar d_{i,j}(k)$ and $\bar d_{i',j}(k)$, appended to $a_i$ and $a_{i'}$ respectively, come from an instance of Case $2$, and are therefore pairwise distinct. Likewise, condition $3$ is satisfied. Since we have not interfered with the tuples $\bar c_{i,j}$ in the above process and these contain pairwise distinct elements by \Cref{lem:bipartite}, Condition $4$ is also satisfied. Finally, Condition $5$ is trivially satisfied since the elements remaining in $\bar d_{i,j}$ fall into Case $4$. 
\end{proof}

With this, we can prove the main theorem of this section.

\begin{theorem}\label{prop:paths}
Let $\C$ be a monotone class of $\L$-structures such that $\Gaif(\C)$ is somewhere dense. Then there is a path formula $\phi(\bar x,\bar y)=\exists \bar w \psi(\bar x,\bar y,\bar w)$ with parameters $\bar p$ and for each bipartite graph $G=(U,V;E) \in \Bfrak$ there is some $M_G \in \C$ and sequences of tuples $(\bar a_u)_{u \in U}$ $(\bar b_v)_{v \in V},$ $(\bar h_{u,v})_{(u,v) \in E}$ from $M_G$ such that:
\begin{enumerate}
    \item $M_G \models \phi(\bar a_u, \bar b_v)$ if, and only if, $(u,v) \in E$ (so, in particular $\C$ is not NIP);
    \item If $(u,v) \in E$ then $M_G \models \psi(\bar a_u,\bar b_v, \bar h_{u,v})$;
    \item The equality type of $\bar p_{u,v}=\bar a_u^\frown \bar b_v^\frown \bar h_{u,v}$ is constant for all $(u,v) \in E(G)$;
    \item Any two tuples in $\{\bar a_u, \bar b_v, \bar h_{u,v}: u \in U, v\in V \}$ are disjoint and do not intersect the parameters $\bar p$.
\end{enumerate}
\end{theorem}

\begin{proof}

   Let $\phi(x,y,\bar z)=\exists \bar w \psi(x,y,\bar z,\bar w)$, with parameters $\bar p$, and $(M_n)_{n \in \N}$ be as in \Cref{lem:disjoint}. For clarity, we again write $(a^n_i)_{i \in [n]}, (b^n_j)_{j \in [n]}, (\bar c^n_{i,j})_{(i,j) \in [n]^2}$ to denote the elements from that lemma coming from $M_n$. Consider the tuples $\bar p^n_{i,j}=\bar a^n_i\ ^\frown \bar b^n_j \ ^\frown \bar c^n_{i,j}\ ^\frown \bar d^n_{i,j}$, and let $q=|\bar p^n_{i,j}|$. Observe that for every $n \in \N$, at most $q\cdot |\bar p|$ many tuples $\bar p^n_{i,j}$ intersect the parameters $\bar p$ because of the conditions in Lemma \ref{lem:disjoint}. By working with suitably large $n$ and avoiding these tuples, we may relabel so that no $\bar p^n_{i,j}$ intersects $\bar p$. 

    For $i,j,k,l \in [n]$, we say that the tuples $\bar p^n_{i,j}$ and $\bar p^n_{k,l}$ \emph{intersect trivially} whenever
 \begin{center}$
   \bar p_{i,j}\cap \bar p_{k,l} = 
    \begin{cases}
        \bar p_{i,j}, &\quad \text{if } i=k \land j=l \\
       \bar a_i, &\quad \text{if } i=k \land j\neq l  \\
       \bar b_j, &\quad \text{if } i\neq k \land j = l \\
       \emptyset,&\quad \text{otherwise. }  \\
       
     \end{cases}$
     \end{center}

    Letting $f(n)=q \cdot (n-1)^2 + n$, we claim that for all $n \in \N$ and all $m \geq f(n)$ we may find a set $A_n\subseteq [f(n)]$ of size $n$ so that $\bar p^m_{i,j}$ and $\bar p^m_{k,l}$ intersect trivially for all $i,j,k,l \in A_n$.

    We show this by induction. Indeed, for $n=1$ this is trivially true as $A_1=[1]$ works for all $m \geq 1$. Suppose that the claim holds for $n-1$ and fix $m \geq f(n)$. Since $f(n)\geq f(n-1)$, by the induction hypothesis there is some $A_{n-1} \subseteq [f(n-1)]\subseteq [f(n)]$ of size $n-1$ so that $\bar p^m_{i,j}$ and $\bar p^m_{k,l}$ intersect trivially for all $i,j,k,l \in A_{n-1}$.
    Notice, that because of Lemma \ref{lem:disjoint}, for every fixed $\bar p_{i,j}$, there are at most $q$ tuples $\bar p_{k,l}$ that do not intersect trivially with it. Hence, there are at most $q \cdot (n-1)^2$ elements $l \in [f(n)]$ such that $\bar p^m_{i,j}$ and $\bar p^m_{k,l}$ do not intersect trivially for all $i,j,k \in A_{n-1}$. Since $[f(n)]$ contains an additional $n$ elements, we are guaranteed to find some $l \in [f(n)]$, which is not one of the $n-1$ elements of $A_{n-1}$, such that $\bar p^m_{i,j}$ and $\bar p^m_{k,l}$ intersect trivially for all $i,j,k \in A_{n-1}$. We may therefore let $A_n=A_{n-1}\cup\{l\}$. 

    Hence,  we may consider the subsequence $(M_{f(n)})_{n \in \N}$ and relabel the tuples appropriately, so that all tuples $\bar p^n_{i,j}, \bar p^n_{k,l}$ intersect trivially for all $n \in \N$ and $i,j,k,l \in [n]$. Furthermore, by an application of \Cref{th:partramsey}, we may assume that the tuples $\bar p^n_{i,j}$ have the same equality type for all $i,j \in [n]$ and all $n \in \N$. More precisely, for every pair $(i,j)\in [n]^2$ let $\Delta_n(i,j):=\Delta_=(\bar p^n_{i,j})$. Letting $q=|\bar p_{i,j}|$, it is easy to see that there are at most $p=2^{q^2}$ sets $\Delta_n(i,j)$. It follows by Theorem \ref{th:partramsey}, that there are subsets $A,B$ of $[\Pcal(n,2,p)]$ of size $n$ such that $\Delta_{\Pcal(n,2,p)}(i,j)$ is constant for all $i \in A$, $j \in B$. Hence, we may relabel appropriately so that $\Delta_n(i,j)$ is constant for all $i,j \in [n]$. Since there are only finitely many such sets, the pigeonhole principle implies that we may prune the sequence $(M_n)_{n \in \N}$ so that $\Delta_n(i,j)$ is uniformly constant for all $n \in \N$.

   It follows that no tuple $\bar a^n_i$ can intersect a tuple $\bar b^n_j$. Indeed, since the equality types are constant, and in particular $\Delta_n(i,j)=\Delta_n(i,j')$, if $\bar a^n_i$ and $\bar b^n_j$ had an element in common then $\bar b^n_j(k)=\bar b^n_{j'}(k)$ for some $k$ and all $j'\neq j$, contradicting the assumptions of Lemma \ref{lem:disjoint}. Likewise, no tuple $\bar h^n_{i,j}=\bar c^n_{i,j}\ ^\frown \bar d^n_{i,j}$ can intersect the tuples $\bar a_i$ or $\bar b_j$. Since the tuples $\bar p^n_{i,j}$ intersect trivially, this implies that any two tuples $\{\bar a^n_i, \bar b^n_j, \bar h^n_{i,j}:i,j \in A_n\}$ are pairwise disjoint, and furthermore do not intersect the parameters $\bar p$.

    For every $n \in \N$, consider the weak substructure $M'_n \leq M_n$ consisting of the elements in $\bar p^n_{i,j}$ and the parameters $\bar p$, and containing solely the relations necessary to witness $M_n\models \psi(\bar a^n_i,\bar b^n_j,\bar h^n_{i,j})$. By monotonicity, $M'_n \in \C$. Notice that every tuple appearing in a relation of $M'_n$ contains at least one element of $\bar c^n_{i,j}$ for some $i,j \in [n]$. Indeed, the elements of $\bar c_{i,j}$ correspond precisely to the joints of the paths $\phi(\bar a_i,\bar b_j)$, and since $\Mcal_\phi$ has length $\geq 2$ every path has at least one joint.
    
    Finally, given $G=(U,V;E)$ with $U=V=[n]$, let $M_G \in \C$ be the induced substructure of $M'_n$ obtained by removing $\bar h^n_{i,j}$ for all $(i,j) \not\in E$. Since the tuples in $\{ \bar a^n_i,\bar b^n_j,\bar h^n_{i,j} : i,j \in [n] \}$ are pairwise disjoint, it follows that $\bar h^n_{i,j}\in M_G$ for $(i,j) \in E(G)$. Hence, letting $\phi'(\bar x, \bar y)=\exists \bar z \phi(\bar x,\bar y,\bar z)$, we see that $M_G \models \phi'(\bar a^n_i,\bar b^n_j)$ for all $(i,j) \in E(G)$. Moreover, $M_G \models \neg\phi(\bar a_i,\bar b_j)$ for $(i,j) \notin E(G)$. Indeed, since the elements of $\bar c^n_{i,j}$ are not in $M_G$ for $(i,j) \notin E(G)$, the above observation implies that $M_G \models \neg\phi(\bar a_i,\bar b_j)$. 
\end{proof}

    Note that all of the above can be proved by working with an appropriate infinite model of $\Th(\C)$ obtained by compactness, and applying the infinite versions of the different Ramsey theorems. We have chosen to give a finitistic proof, which is admittedly more involved, so that everything is carried out effectively. Therefore, if we assume that the \emph{VC-dimension} of formulas in the class is computable, we may compute given $r$ the maximum size of an $r$-subdivided clique occurring in the Gaifman graph of a structure in $\C$. 

\begin{definition}
We say that a class $\C$ of structures is effectively NIP if there is a computable function $f:\N \to \N$ such that for all formulas $\phi(\bar x,\bar y)$ and all structures $M \in \C$ there is no $n > f(|\phi|)$ and $(\bar a_i)_{i \in [n]}, (\bar b_J)_{J \subseteq [n]}$ with 
\[M \models \phi(\bar a_i,\bar b_J) \iff i \in J.\]
\end{definition}

Recall that we call a class $\C$ of graphs \emph{effectively nowhere dense} whenever there is a computable function $f:\N \to \N$ such that for all $r \in \N$ and for all $G \in \C$ we have that $K^r_{f(r)}$ is not a subgraph of $G$. 

\begin{corollary}
Let $\C$ be a monotone and (monadically) NIP class of $\L$-structures in a finite relational language. Then $\Gaif(\C)$ is nowhere dense. Moreover, if $\C$ is effectively NIP then $\Gaif(\C)$ is effectively nowhere dense. 
\end{corollary}

\section{Intractability}  \label{sec:tractable}

In this section, we prove that any monotone class of relational structures whose Gaifman class is somewhere dense polynomially interprets the class of all bipartite graphs, and is therefore intractable. Towards this, we first strengthen Theorem \ref{prop:paths} to obtain a simple path formula $\phi$ as well as a computable function $\Phi:\Bfrak \to \C$ such that $\phi$ codes the edge relation of $G$ in $\Phi(G)$.

\begin{lemma}{simplepath}\label{cl:simplepath}
Let $\C$ be a monotone class of $\L$-structures such that $\Gaif(\C)$ is somewhere dense. Then there is a simple path formula $\phi(\bar x,\bar y)$ with parameters $\bar p$ and a polynomial time computable function $\Phi:\Bfrak \to \C$, such that for each bipartite graph $G=(U,V;E) \in \Bfrak$ there are tuples $(\bar a_u)_{u \in U}$ $(\bar b_v)_{v \in V},$ $(\bar h_{u,v})_{(u,v) \in E}$ from $\Phi(G)$ satisfying:
\[\Phi(G) \models \phi(\bar a_u, \bar b_v) \text{ if, and only if, }(u,v) \in E.\]
Furthermore, the interpretation of the parameters $\bar p$ in $\Phi(G)$ can be computed in constant time from $G \in \Bfrak$.
\end{lemma}

\begin{proof}
Let $\phi$ and $(M_G)_{G \in \Bfrak}$ be as in Theorem \ref{prop:paths}. Consider the path $\Mcal_\phi$. Observe that either there is a step $\bar e_i$ such that both $\bar e_i \cap \bar x = \bar x'\neq \emptyset$ and $\bar e_i \cap \bar y = \bar y' \neq \emptyset$, or there are $i<j$ and steps $\bar e_i$, $\bar e_j$ such that $\bar e_i \cap \bar y =\emptyset, \bar e_j \cap \bar x= \emptyset$ and $\bar e_i \cap \bar x = \bar x' \neq \emptyset, \bar e_j \cap \bar y = \bar y' \neq \emptyset$ and for all $k\in\{i+1,\dots,j-1\}$ we have that $\bar e_k\cap \bar x = \bar e_k\cap\bar y = \emptyset$. Consider the induced substructure $\Mcal'$ of $\Mcal_\phi$ consisting solely of the step $\bar e_i$ in the first case or the steps $\bar e_i, \dots, \bar e_j$ in the second, and let $\phi'(\bar x',\bar y')=\exists \bar w' \psi'(\bar x',\bar y', \bar w')$ be the canonical formula of $(\Mcal',\bar x',\bar y')$. Clearly, $\phi'$ is a simple path formula, and it follows by construction that for each $G \in \Bfrak$ we may pick minimal subtuples $\bar a'_u \subseteq \bar a_u, \bar b'_v \subseteq \bar b_v, \bar c'_{u,v} \subseteq \bar c_{u,v}\bar h'_{u,v} \subseteq \bar h_{u,v} \in M_G$ for all $u \in U, v \in V$ such that :
\begin{itemize}
    \item     $M_G \models \phi'(\bar a'_u,\bar b'_v)$ if, and only if, $(u,v) \in E$, and
    \item     $(u,v) \in E$ implies $M_G \models \phi'(\bar a'_u,\bar b'_v,\bar c'_{u,v},\bar h'_{u,v}).$
\end{itemize}
Clearly, these new subtuples are mutually disjoint and do not intersect any of the parameters $\bar p'\subseteq \bar p$ that appear in $\phi'$. We finally let $M'_G$ be the induced substructure of $M_G$ consisting solely of these subtuples. Since the equality type of all tuples $\bar p'_{u,v}=\bar a_u^\frown \bar b_v ^\frown \bar h_{u,v}$ is uniformly constant by Theorem \ref{prop:paths}, it follows that $M_G$ may be computed from $G=(U,V;E)$ by adding disjoint tuples $(\bar a^G_u)_{u \in U}, (\bar b^G_v)_{v \in V}, (\bar h^G_{u,v})_{(u,v)\in E(G)}, \bar p^G$ of appropriate equality types to represent vertices and existential witnesses, and the relations specified by $\phi'$ to represent the edges.  Clearly, the tuple $\bar p^G$ which interprets the parameters of $\phi'$ is obtained in constant time from $G$.
\end{proof}

\begin{theorem}\label{thm:intractability}
Let $\C$ be a monotone class of $\L$-structures such that $\Gaif(\C)$ is somewhere dense, and assume that $\mathsf{AW}[*]\neq \FPT$. Then $\FO$ model-checking on $\C$ is not fixed-parameter tractable. 
\end{theorem}

\begin{proof}
Let $\C$ satisfy the above, and assume that $\mathsf{AW}[*]\neq \FPT$. We argue that we may polynomially interpret the class of all bipartite graphs in $\C$. 

Let $\phi(\bar x,\bar y)$ be the simple path formula from \Cref{cl:simplepath}. Without loss of generality, we may assume that $\phi$ is not symmetric (in the sense of Definition \ref{def:symmetric}). Indeed, if $\phi$ is symmetric let $\sigma \in S_n$ be the non-identity permutation from Definition \ref{def:symmetric}, and consider the formula $\phi'(\bar x,\bar y)=\phi(\bar x,\sigma^{-1}(\bar y))$, where $\sigma^{-1}$ is applied to the indices of $\bar y$. Clearly, $\phi'$ is no longer symmetric, while the tuples $(\bar a_u)_{u \in U}$ $(\sigma(\bar b_v))_{v \in V},$ $(\bar h_{u,v})_{(u,v) \in E}$ still satisfy the conditions in Lemma \ref{cl:simplepath}. 

Now, let $k$ the length of the path $\Mcal_\phi$ and define the auxiliary map:
\[\begin{aligned}
f: \Bfrak&\to\Bfrak\\
G=(U,V;E)&\mapsto (U^\prime, V^\prime; E^\prime),    
\end{aligned}\]
where $U^\prime := U \sqcup \{\dot{u}_{v,1},\dots,\dot{u}_{v,k+1}:v\in V\}$, $V^\prime := V \sqcup \{\dot{v}_{u,1},\dots,\dot{v}_{u,k+1}:u\in U\}$, and $E^\prime := E\sqcup\{(u,\dot{v}_{u,i}):u\in U,i\in[k+1]\}\sqcup\{(v,\dot{u}_{v,i}):v\in V,i\in[k+1]\}$.
    
This is clearly computable in polynomial time. Given $G=(U,V;E) \in \Bfrak$, consider $\Phi \circ f(G) \in \C$ given from \Cref{prop:paths}, and let:
\begin{center}
    $\theta_U(\bar x) := \exists^{>k} \bar y \phi(\bar x,\bar y) \land \bar x \neq \bar p \text{ and } \theta_V(\bar y) := \exists^{>k} \bar x \phi(\bar x,\bar y) \land \bar y \neq \bar p,  $
\end{center}

where $\bar p$ are the parameters of $\phi$. Without loss of generality, we may assume that $|\bar x|=|\bar y|$, for if $m=|\bar y| < |\bar x|=n$, then we may take $\theta_V(\bar y,y_{m+1},\dots,y_n)$ to be $\theta_V(\bar y)\land \bigwedge_{i = m}^{n-1} (y_i = y_{i+1})$, and similarly if $|\bar x|<|\bar y|$. So, let $\theta(\bar x) = \theta_V(\bar x) \lor \theta_U(\bar x)$.

Observe that $G$ is an induced subgraph of $f(G)$, so we may view $\Phi(G)$ as an induced substructure of $\Phi\circ f(G)$. Letting $\bar p_{u,v}=\bar a_u ^\frown \bar b_v ^\frown\bar h_{u,v}$, it holds that $\bar p_{u,v}\cap\bar p =\emptyset$ and 
\begin{center}
    $\bar p_{u,v} \cap \bar p_{u',v'} = 
    \begin{cases}
        \bar a_{u} & \text{ if } u = u^\prime; \\
        \bar b_{v} & \text{ if } v = v^\prime; \\
        \emptyset  & \text{ otherwise}.
    \end{cases}$
\end{center}

whenever $(u,v)\neq (u^\prime,v^\prime)$. Hence, the only non-parameter elements that appear more than $k$ times within a path are those in the tuples $\bar a_u$ and $\bar b_v$ for $u \in U$ and $v \in V$, i.e.\ those tuples corresponding to the elements of $G$. Since $\phi$ is not symmetric, it follows that $\theta(\Phi \circ f(G))=\{\bar a_u, \bar b_v : u \in U, v \in V\}$, and so the pair $I=(\theta(\bar x),\phi(\bar x,\bar y))$ is an interpretation with computable parameters such that $I(\Phi \circ f(G))=G$ for all $G \in \Bfrak$. It follows that $\Bfrak \leq_P \C$, and therefore $\C$ is not tractable. 
\end{proof}

\section{From IP to somewhere density}\label{sec:dense}

Combinatorial configurations that perform pairing on tuples are central to the analysis of monadically NIP classes of structures, and more broadly, monadically NIP first-order theories. Surprisingly, it turns out that monadic NIP is equivalent to the absence of such configurations (see Theorem 4.1 in \cite{mNIP}). In particular, any non-monadically NIP class of structures admits a tuple-coding configuration. Restricting to monotone (resp.\ hereditary) classes guarantees that the presence of these configurations is achieved via  ``low complexity'' formulas, namely quasi-positive (resp.\ existential) formulas.  We make the above precise in the following definition and subsequent theorem. 

\begin{definition}
Let $M$ be an (infinite) $\L$-structure. A \emph{pre-coding configuration} in $M$ consists of a formula $\phi(\bar x, \bar y, z)$ with parameters, a sequence $\Ical = \langle \bar d_i : i \in \N \rangle$, and $\{ c_{s,t} : s, t \in \N \}$ such that the following hold for all $s, t \in \N$:
\begin{enumerate}
    \item $M \models \phi(\bar d_s, \bar d_t, c_{s,t}) $;
    \item $M \models \neg\phi(\bar d_s,\bar d_v, c_{s,t})$, for all $v> t$;
    \item $M \models \neg(\bar d_u,\bar d_t, c_{s,t})$, for all $u<s$.
\end{enumerate}

\end{definition}

Furthermore, when $\phi(\bar x, \bar y, z):=\exists \bar w\psi(\bar x,\bar y, z,\bar w)$ is existential in the definition above, we say that the pre-coding configuration is \emph{disjoint} if there are tuples $\bar h_{s,t}$ from $M$ such that $M\models\psi(\bar d_s,\bar d_t,c_{s,t},\bar h_{s,t})$ and any two of $\{\bar d_s, c_{s,t}, \bar h_{s,t}: s, t \in \N \}$ are disjoint and they do not intersect the parameters of $\phi$. Note that we may always transform a pre-coding configuration by an existential formula into a disjoint one via a canonical Ramsey-type argument, akin to the proof of Lemma \ref{lem:disjoint}. The following theorem guarantees that we may find quasi-positive pre-coding in any monotone class which is not monadically NIP. 

\begin{theorem}[Braunfeld, Laskowski, \cite{exmNIP}]\label{xmNIP}
Let $\C$ be a monotone class of structures that is not monadically NIP. Then $\C$ admits disjoint pre-coding via a quasi-positive formula, i.e.\ there is some $M \models \Th(\C)$ and a disjoint pre-coding configuration in $M$ witnessed by a quasi-positive formula $\phi(\bar x, \bar y, z):=\exists \bar w\psi(\bar x,\bar y, z,\bar w)$. 
\end{theorem}

The next lemma tells us that whenever a quasi-positive formula does disjoint pre-coding in a structure $M$, then it must create paths in $\Gaif(M)$ which avoid parameters.

\begin{restatable}{lemma}{precodingpaths}\label{precodingpaths}
Let $M$ be an $\L$-structure with a disjoint pre-coding configuration given by a quasi-positive formula $\phi(\bar x, \bar y, z)$. Then there are $x_0 \in \bar x$ and $y_0 \in \bar y$ such that there is a path from $x_0$ to $y_0$ in $\Gaif(\Mcal_{\tilde\phi})$ that does not pass through the parameters of $\phi$.     
\end{restatable}

\begin{proof}
Let $\phi(\bar x,\bar y,z)=\exists \bar w \psi(\bar x,\bar y,z,\bar w)$, $M$ and $(\bar d_s)_{s \in \N}, (c_{s,t})_{s,t \in \N}, (\bar h_{s,t})_{s,t \in \N}$ witness a disjoint quasi-positive pre-coding configuration in $M$. Consider the formula $\tilde\psi$, i.e.\ the quantifier-free formula obtained by removing all conjuncts of the form $v_i\neq v_j$ from $\psi$, and let $\tilde \phi(\bar x,\bar y,z)$. Let $G = \Gaif(\Mcal_{\tilde\phi})\setminus \{ p : p$ is a parameter in $\phi \}$, $M_x = N^G(\bar x)$, and $M_{y,z}=N^G(\bar y, z)$, i.e.\ $M_x$ (resp. $M_{y,z}$) is the set of those variables that are reachable from $\bar x$ (resp.\ $\bar y\cup\{z\}$) in $\Gaif(\Mcal_{\tilde\phi})$ by avoiding parameters. Assume for a contradiction that $M_x$ and $M_{y,z}$ are disjoint, and without loss of generality write $\bar w = \bar w_x^\frown \bar w_{y,z}$ with $\bar w_x \in M_x$ and $\bar w_{y,z} \in M_{y,z}$, so that $\phi(\bar x,\bar y,z)=\exists \bar w_x \exists \bar w_{y,z} \psi(\bar x,\bar y,z,\bar w_x,\bar w_{y,z})$. Similarly, partition the existential witnesses $\bar h_{i,j}$ from $M$ into $\bar \alpha_{i,j}$ and $\bar \beta_{i,j}$, so that $\bar h_{i,j}=\bar \alpha_{i,j}^\frown \bar \beta_{i,j}$, and $M \models \psi(\bar d_i,\bar d_j, c_{i,j},\bar \alpha_{i,j},\bar \beta_{i,j})$ for all $i,j \in \N$. 

By assumption, $M \models \psi(\bar d_1,\bar d_2, c_{1,2},\bar \alpha_{1,2},\bar \beta_{1,2})\land \phi(\bar d_2,\bar d_3, c_{2,3},\bar \alpha_{2,3},\bar \beta_{2,3})$, and so in particular $M \models \tilde\phi(\bar d_1,\bar d_2, c_{1,2},\bar \alpha_{1,2},\bar \beta_{1,2})\land \tilde\phi(\bar d_2,\bar d_3, c_{2,3},\bar \alpha_{2,3},\bar \beta_{2,3})$. Since $M_x$ and $M_{y,z}$ are disjoint, Lemma \ref{ppcomponents} implies that $M \models \tilde\psi(\bar d_1, \bar d_3, c_{2,3},\bar \alpha_{1,2},\bar \beta_{2,3})$. However, the fact that $\phi$ is a pre-coding configuration implies that $M \models \neg \phi(\bar d_1, \bar d_3, c_{2,3})$. Consequently, it must be that there are variables $u \in M_x$ and $v \in M_{y,z}$ such that $\psi$ contains $u \neq v$ as a conjunct. Since $\M \models \neg \psi(\bar d_1, \bar d_3, c_{2,3},\bar \alpha_{1,2},\bar \beta_{2,3})$, this means that the tuples $\bar d_1^\frown \bar \alpha_{1,2}$ and $\bar d_3, c_{2,3},\beta_{2,3}$ must have an element in common, contradicting disjointedness. 

It follows that there is some $x_0 \in \bar x$ such that there is a path avoiding parameters from $x_0$ to some element in $M_{\bar y,z}$; this is either some $y_1 \in \bar y$, and so we are done, or $z$. By a similar argument, we may find some $y_0 \in \bar y$ such that there is a path avoiding parameters to some element in $M_{\bar x, z}$; either this is some $x_1 \in \bar x$, and so we are done, or this $z$, in which case $x_0$ and $y_0$ are connected via $z$ by a path that avoids parameters. 
\end{proof}

The above lemma, together with the fact that the elements in a pre-coding configuration can be taken to suitably disjoint, is what allows us to find subdivided cliques in $\Gaif(\C)$. 

\begin{theorem}
Let $\C$ be a monotone class of $\L$-structures that is not monadically NIP. Then $\Gaif(\C)$ is somewhere dense.
\end{theorem}

\begin{proof}
If $\C$ is a monotone class of $\L$-structures that is not monadically NIP, then $\C$ admits disjoint coding by a quasi-positive formula $\phi(\bar x, \bar y, z)=\exists\bar w\psi(\bar x,\bar y, z,\bar w)$. Let $M \models \Th(\C)$ and $\{\bar d_s, c_{s,t}, \bar h_{s,t}: s, t \in \N \}$ witness this as in Theorem \ref{xmNIP}. By Lemma \ref{precodingpaths}, this implies that there are $a_s\in \bar d_s, b_t \in \bar d_t$ for $s, t \in \N$ with the property that $a_s$ is connected to $b_t$ in $\Gaif(M)$ by a path of some fixed length $r$. Moreover, the paths are disjoint for different $(s,t)$ as they do not pass through the parameters of $\phi$ and non-parameter tuples are disjoint. It follows that $K^r_{\omega,\omega} \leq \Gaif(M)$. Hence, we obtain for each $n \in \N$ some $M_n \in \C$ such that $K^r_{n,n}\leq \Gaif(M_n)$, and so it follows that $\Gaif(\C)$ is somewhere dense. 
\end{proof}

Since being nowhere dense is preserved by taking the monotone closure, this has the following stronger implication. 

\begin{corollary}
Let $\C$ be any class of $\L$-structures such that $\Gaif(\C)$ is nowhere dense. Then $\C$ is monadically stable.
\end{corollary}
\begin{proof}
Consider $\C$ as above, and let $\tilde{\C}$ be its monotone closure. It follows that $\Gaif(\tilde{\C})$ is nowhere dense, as it is contained in the monotone closure of $\Gaif(\C)$ which is nowhere dense by assumption. By the above, it follows that $\tilde{\C}$ is monadically NIP, and so by monotonicity it is monadically stable. As $\C\subseteq \tilde{\C}$, this implies that $\C$ is monadically stable.
\end{proof}

None of the results here requires that $\C$ contains finite structures, nor that $\L$ is finite. Hence, by simply taking $\C=\{M\}$ for any infinite $\L$-structure in a possibly (countably) infinite relational language, the above is a generalisation of the original result of Podewski and Ziegler \cite{podewskiziegler}.

\begin{corollary}
Let $M$ be a relational structure such that for all $r \in \N$ there is some $n \in \N$ with $K^r_n \not\leq \Gaif(M)$. Then $M$ is monadically stable.
\end{corollary}

\section{Conclusion}

Our paper settles the question of Adler and Adler, showing that tameness for a monotone class of relational structures can be completely recovered from the structural sparsity of its Gaifman class. We believe that many results from the theory of sparse graphs will generalise to relational structures by working with the Gaifman class, and we plan to exhibit such generalisations in future work. 

Although this has not been addressed thus far, monotonicity as defined for classes of relational structures does not fully correspond to monotonicity in the standard graph-theoretic sense. Indeed, in the graph-theoretic sense, a monotone class of graphs is one closed under removal of \emph{undirected} edges, that is, simultaneous removal of pairs of relations $E(u,v), E(v,u)$. However, a monotone class of $\{E\}$-structures is one where we can remove any $E$ relation (so possibly we can turn an undirected edge into a directed one). In future work, we aim to address this subtle difference by introducing \emph{symmetrically monotone} classes, so that our results can extend to broader classes of relational structures, such as classes of undirected hypergraphs closed under removal of hyperedges. 

Finally, our paper makes a significant contribution towards \Cref{conj}, settling it for the case of monotone classes of structures. While the machinery used in this paper will certainly assist in tackling the full conjecture, we believe that new techniques are required for this task. Here, it is important to understand the role of linear orders in the collapse of monadic NIP and bounded twin-width for hereditary classes of ordered graphs, and to identify which model-theoretic conditions generalise this phenomenon to arbitrary hereditary graph classes. 

\bibliographystyle{plain}
\bibliography{bibliography}

\begin{thebibliography}{10}

\bibitem{AdlerAdler2014}
Hans Adler and Isolde Adler.
\newblock Interpreting nowhere dense graph classes as a classical notion of
  model theory.
\newblock {\em European Journal of Combinatorics}, 36:322--330, 2014.

\bibitem{Bal17}
John~T. Baldwin.
\newblock {\em Fundamentals of Stability Theory}.
\newblock Perspectives in Logic. Cambridge University Press, 2017.

\bibitem{tww4}
Édouard Bonnet, Ugo Giocanti, Patrice~Ossona de~Mendez, Pierre Simon, Stéphan
  Thomassé, and Szymon Toruńczyk.
\newblock Twin-width iv: ordered graphs and matrices, 2021.

\bibitem{mNIP}
Samuel Braunfeld and Michael Laskowski.
\newblock Characterizations of monadic \text{NIP}.
\newblock {\em Transactions of the American Mathematical Society, Series B},
  8(30):948--970, 2021.

\bibitem{exmNIP}
Samuel Braunfeld and Michael~C. Laskowski.
\newblock Existential characterizations of monadic \text{NIP}, 2022.

\bibitem{CM77}
Ashok~K. Chandra and Philip~M. Merlin.
\newblock Optimal implementation of conjunctive queries in relational data
  bases.
\newblock In John~E. Hopcroft, Emily~P. Friedman, and Michael~A. Harrison,
  editors, {\em Proceedings of the 9th Annual {ACM} Symposium on Theory of
  Computing, May 4-6, 1977, Boulder, Colorado, {USA}}, pages 77--90. {ACM},
  1977.

\bibitem{Daw07}
Anuj Dawar.
\newblock Finite model theory on tame classes of structures.
\newblock In {\em MFCS}, volume 4708 of {\em Lecture Notes in Computer
  Science}, pages 2--12. Springer, 2007.

\bibitem{erdos}
Paul Erd\H{o}s and Richard Rado.
\newblock A combinatorial theorem.
\newblock {\em Journal of the London Mathematical Society}, 1(4):249--255,
  1950.

\bibitem{GPT22}
Jakub Gajarsk{\'{y}}, Michal Pilipczuk, and Szymon Torunczyk.
\newblock Stable graphs of bounded twin-width.
\newblock In Christel Baier and Dana Fisman, editors, {\em {LICS} '22: 37th
  Annual {ACM/IEEE} Symposium on Logic in Computer Science, Haifa, Israel,
  August 2 - 5, 2022}, pages 39:1--39:12. {ACM}, 2022.

\bibitem{nwdtract}
Martin Grohe, Stephan Kreutzer, and Sebastian Siebertz.
\newblock Deciding first-order properties of nowhere dense graphs.
\newblock {\em J. ACM}, 64(3), jun 2017.

\bibitem{hodges}
Wilfrid Hodges.
\newblock {\em Model Theory}.
\newblock Encyclopedia of Mathematics and its Applications. Cambridge
  University Press, 1993.

\bibitem{KMV_2017}
Alexandr Kostochka, Dhruv Mubayi, and Jacques Verstra\"ete.
\newblock Tur\'an problems and shadows \text{II: T}rees.
\newblock {\em Journal of Combinatorial Theory, Series B}, 122:457--478, 2017.

\bibitem{KD09}
Stephan Kreutzer and Anuj Dawar.
\newblock Parameterized complexity of first-order logic.
\newblock {\em Electron. Colloquium Comput. Complex.}, {TR09-131}, 2009.

\bibitem{nowhere}
Jaroslav Ne{\v{s}}et{\v{r}}il and Patrice~Ossona De~Mendez.
\newblock On nowhere dense graphs.
\newblock {\em European Journal of Combinatorics}, 32(4):600--617, 2011.

\bibitem{NOdM12}
Jaroslav Ne{\v{s}}et{\v{r}}il and Patrice~Ossona De~Mendez.
\newblock {\em Sparsity: graphs, structures, and algorithms}, volume~28.
\newblock Springer Science \& Business Media, 2012.

\bibitem{PST18}
Michal Pilipczuk, Sebastian Siebertz, and Szymon Torunczyk.
\newblock Parameterized circuit complexity of model-checking on sparse
  structures.
\newblock In Anuj Dawar and Erich Gr{\"{a}}del, editors, {\em Proceedings of
  the 33rd Annual {ACM/IEEE} Symposium on Logic in Computer Science, {LICS}
  2018, Oxford, UK, July 09-12, 2018}, pages 789--798. {ACM}, 2018.

\bibitem{podewskiziegler}
Klaus-Peter Podewski and Martin Ziegler.
\newblock Stable graphs.
\newblock {\em Fund. Math}, 100(2):101--107, 1978.

\bibitem{ramsey}
F.~P. Ramsey.
\newblock {On a Problem of Formal Logic}.
\newblock {\em Proceedings of the London Mathematical Society},
  s2-30(1):264--286, 01 1930.

\bibitem{She78}
Saharon Shelah.
\newblock {\em Classification theory and the number of nonisomorphic models},
  volume~92 of {\em Studies in Logic and the Foundations of Mathematics}.
\newblock North-Holland Publishing Co., Amsterdam-New York, 1978.

\bibitem{tww}
Pierre Simon and Szymon Toruńczyk.
\newblock Ordered graphs of bounded twin-width, 2021.

\end{thebibliography}

\end{document}